\documentclass[a4paper,11pt]{article}
\usepackage[english]{babel}
\usepackage{epsfig}
\usepackage{psfrag}
\usepackage{amsfonts} 
\usepackage{mathrsfs} 
\usepackage{amssymb, array}  
\usepackage{amsmath,amscd,amsthm, dsfont} 
\usepackage{color}
\usepackage{graphicx, subfigure, xypic}
\numberwithin{equation}{section}
\newtheorem{theorem}{Theorem}[section]
\newtheorem{lemma}[theorem]{Lemma}
\newtheorem{proposition}[theorem]{Proposition}
\newtheorem{definition}[theorem]{Definition}
\newtheorem{corollary}[theorem]{Corollary}

\title{On the coarse geometry of the complex of domains}
\author{Valentina Disarlo}

\begin{document}
\maketitle
\begin{abstract}
The complex of domains $D(S)$ is a geometric tool with a very rich simplicial structure, it contains the curve complex $C(S)$ as a simplicial subcomplex. In this paper we shall regard it as a metric space, endowed with the metric which makes each simplex Euclidean with edges of length 1, and we shall discuss its coarse geometry. We prove that for every subcomplex $\Delta(S)$ of $D(S)$ which contains the curve complex $C(S)$, the natural simplicial inclusion $C(S) \to \Delta(S)$ is an isometric embedding and a quasi-isometry. We prove that, except a few cases, the arc complex $A(S)$ is quasi-isometric to the subcomplex $P_\partial(S)$ of $D(S)$ spanned by the vertices which are peripheral pair of pants, and we prove that the simplicial inclusion $P_\partial(S) \to D(S)$ is a quasi-isometric embedding if and only if $S$ has genus $0$. We then apply these results to the arc and curve complex $AC(S)$. We give a new proof of the fact that $AC(S)$ is quasi-isometric to $C(S)$, and we discuss the metric properties of the simplicial inclusion $A(S) \to AC(S)$. 
\end{abstract}

\section{Introduction}
Let $S=S_{g,b}$ be a compact orientable surface of genus $g$ with $b$ boundary components. A simple closed curve on $S$ is \emph{essential} if it is not null-homotopic or homotopic to a boundary component. An \emph{essential annulus} on $S$ is a regular neighbourhood of an essential curve. The well-known \emph{complex of curves}\index{complex! of curves} is a simplicial complex whose simplices are defined in the following way: for $k \geq 0$ any collection of $k+1$ distinct isotopy class of essential annuli, which can be realized disjointly on $S$, spans a $k$-simplex. This complex has been introduced by Harvey in \cite{Har}. Ivanov, Korkmaz and Luo proved that except for a few surfaces the action of the extended mapping class group of $S$ on $C(S)$ is rigid, i.e. the automorphism group of $C(S)$ is the extended mapping class group of $S$ (see \cite{Iv, Kork, Luo}). Like any other simplicial complex here mentioned, the curve complex can be equipped with a metric such that each simplex is an Euclidean simplex with sides of length 1. This metric is natural in the sense that the simplicial automorphism group acts by isometries with respect to it. As a metric space, $C(S)$ is quasi-isometric to its 1-skeleton. Masur and Minsky proved that the complex of curves $C(S)$ is Gromov hyperbolic (see \cite{MM1}).

In this paper, we shall deal with the coarse geometry of some sort of ``generalized'' curve complex, the so-called \emph{complex of domains}\index{complex! of domains} $D(S)$. A \emph{domain}\index{domain} $D$ on $S$ is a connected subsurface of $S$ such that each boundary component of $\partial D$ is either a boundary component of $S$ or an essential curve on $S$. Pairs of pants of $S$ and essential annuli are the simplest examples of domains of $S$. The complex of domains $D(S)$, introduced by McCarthy and Papadopoulos (see \cite{MCP}), is defined as follows: for $k \geq 0$, its $k$-simplices are the collections of $k+1$ distinct isotopy classes of domains which can be realized disjointly on $S$. It has a very rich simplicial structure, which encodes a large amount of information about the structure of $S$. There is a simplicial inclusion $C(S) \to D(S)$. It is then natural to ask whether this and other ``natural'' simplicial maps between subcomplexes of $D(S)$ (see \cite{MCP}) have any interesting property when regarded as maps between metric spaces, and we shall deal with this question. We shall also deal with the relation between $D(S)$ and the similarly defined \emph{arc complex}\index{complex! of arcs} $A(S)$. The $k$-simplices of $A(S)$ are the collections of $k+1$ isotopy classes of disjoint essential arcs on $S$ for every $k \geq0$. Although this definition is very close to the one of $D(S)$, there is no ``natural'' simplicial inclusion $A(S) \to D(S)$. \\ 

The paper is organized as follows: Section 2 contains some generalities about the complex of domains and the proof that any simplicial inclusion of $C(S)$ in a subcomplex $\Delta(S)$ of $D(S)$ induces a quasi-isometry $C(S) \to \Delta(S)$. In Section 3 we describe a quasi-isometry between $A(S)$ and the subcomplex $P_\partial(S)$ of $D(S)$ spanned by peripheral pairs of pants, and we prove that the simplicial inclusion $P_\partial(S) \to D(S)$ is a quasi-isometric embedding if and only if $S$ has genus $0$. In Section 4 we combine the previous results to prove that $AC(S)$ is quasi-isometric to $C(S)$ and to show that the simplicial inclusion $A(S) \to AC(S)$ is a quasi-isometric embedding if and only if $S$ has genus $0$. \\

\textbf{Acknowledgments - } I would like to thank Athanase Papadopoulos, Mustafa Korkmaz and Daniele Alessandrini for helpful remarks. 

\section{The complex of domains $D(S)$ and its subcomplexes} 
Let $S=S_{g,b}$ be a connected, orientable surface of genus $g$ and $b$ boundary components. Let $c(S)=3g+b-4$ be the \emph{complexity} of the surface. Recall that $c(S) = 0$ if and only if $S$ is different from a sphere with at most four holes or a torus with at most one hole. In this cases $C(S)$ is either empty or not connected. We shall then always assume $c(S) > 0$.

We recall a few facts about the complex of domains. For further details, the reader can consult \cite{MCP}.\\ 
A \emph{domain} $X$ in $S$ is a proper connected subsurface of $S$ such that every boundary component of $X$ is either a boundary component of $S$ or an essential curve on $S$.   

\begin{definition}
The \emph{complex of domains}\index{complex!of domains} $D(S)$ is the simplicial complex such that for all $k \geq 0$ its $k$-simplices are the collections of $k+1$ distinct isotopy classes of disjoint domains on $S$.
\end{definition}

\begin{proposition}
If $c(S) > 0$, then $D(S)$ is connected and $\mathrm{dim } D(S) = 5g + 2b - 6$. 
\end{proposition}

There is a natural simplicial inclusion $C(S)\to D(S)$ which sends every vertex of $C(S)$ to the isotopy class of an essential annulus representing it. One simplicial difference between the curve complex $C(S)$ or the arc complex $A(S)$ and the complex of domains $D(S)$ is that maximal (in the sense of inclusion) simplices of $D(S)$ are not necessarily top-dimensional simplices. In the complex of domains, we can find maximal simplices of all dimensions between 1 and $\text{dim } D(S)$. The automorphism group of the complex of domains is closely related to the group of isotopy classes of homeomorphisms of $S$, the so-called \emph{extended mapping class group of} $S$, though in general these two groups do not coincide (fot further details see \cite{MCP}).

\subsection{Subcomplexes of $D(S)$ containing $C(S)$}
We prefer to work in the setting of length spaces rather than the more general setting of metric spaces.
Recall that for $h, k \in \mathbb{R}^+$ an \emph{$(h,k)$-quasi-isometric embedding}\index{embedding!quasi-isometric} between two length spaces $(X, d_X)$ and $(Y, d_Y)$ is a map $f: X \to Y$ such that for every $x,y \in X$ the following holds: $$\frac{1}{h}d_X(x,y) - k \leq d_Y(f(x),f(y)) \leq h d_X(x, y) + k~.$$ 
A \emph{bilipschitz equivalence}\index{bilipschitz equivalence} is a $(h,0)$-quasi-isometric embedding. Recall also that a quasi-isometric embedding is a \emph{quasi-isometry}\index{quasi-isometry} if $Y$ is $c$-dense for some $c>0$, that is, if every point in $Y$ is at distance less or equal to $c$ from some point in $f(X)$. \\
Since we shall be dealing with length metrics on simplicial complexes, when we refer to an arbitrary subcomplex we shall always assume that the subcomplex is connected. Moreover, since the dimension of any complex we shall mention is bounded, we shall always identify the complexes with their $1$-skeleton when we refer to their large scale geometry.

In this section we discuss metric properties of some natural maps between subcomplexes of $D(S)$. We shall use the same notation for geometric objects on the surface (curves, pairs of pants, domains) and their isotopy classes as vertices of $D(S)$. As a first result, we have:

\begin{theorem}\label{incl}
Let $\Delta(S)$ be a subcomplex of $D(S)$ that contains $C(S)$ as a subcomplex. Then the natural simplicial inclusion $i : C(S) \to \Delta(S)$ is an isometric embedding and a quasi-isometry. 
\end{theorem}

\begin{proof}
Let us first prove the theorem in the case $\Delta(S)=D(S)$. \\
For every $c_1$ and $c_2$ in $C(S)$, we have $d_{D(S)}(i(c_1),i(c_2)) \leq d_{C(S)}(c_1,c_2)$. Indeed, let $\sigma$ be a geodesic path on $C(S)$ joining $c_1$ and $c_2$, namely $\sigma$ is given by the edge path $c_1=x_0\cdots x_n=c_2$ such that $d_{C(S)}(x_i, x_{i+1})=1$. By definition $x_i$ and $x_{i+1}$ are represented by two disjoint non-homotopic annuli on $S$, hence $d_{D(S)}(i(x_i),i(x_{i+1}))=1$, and $i(\sigma)$ is a path on $D(S)$ with the same length as $\sigma$.
We thus see that $${d_{D(S)}(i(c_1),i(c_2)) \leq \mbox{Length}_{D(S)}(i(\sigma)) = \mbox{Length}_{C(S)}(\sigma)= d_{C(S)}(c_1,c_2) ~.}$$
Let us now prove the reverse inequality. Let $\Sigma$ be a geodesic segment in $D(S)$ joining $i(c_1)$ and $i(c_2)$, namely $\Sigma$ is given by the edge path $i(c_1)=X_0 \cdots X_{k+1}=i(c_2)$ with $d_{D(S)}(X_i, X_{i+1})=1$ for every $0 \leq i \leq k$. Choose for every vertex $X_i $ a curve $x_i^b$ among the essential boundary components of $X_i$. The condition $X_i \cap X_{i+1} = \varnothing$ implies that either $x_i^b$ is homotopic to $x_{i+1}^b$, or $x_i^b$ and $x_{i+1}^b$ are disjoint. In the first case $x_i^b$ and $x_{i+1}^b$ are represented by the same vertex in the curve complex $C(S)$, in the second case these $x_i^b$ and $x_{i+1}^b$ are represented by two different vertices joined by an edge in $C(S)$.
Then, we can consider the path in $C(S)$ given by the $x_i^b$'s, namely $\Sigma^b:~c_1 x_1^b \cdots x_k^b c_2$, and notice that its length is not greater than the length of $\Sigma$ in $D(S)$. We conclude that
$$d_{C(S)}(c_1,c_2) \leq \mbox{Length}_{C(S)}(\Sigma^b) \leq \mbox{Length}_{D(S)}(\Sigma) = d_{D(S)}(i(c_1),i(c_2)) ~.$$

Now we notice that for an arbitrary $\Delta(S)$, by the above case, for every pair of vertices $c_1,c_2 \in C(S)$ the following holds: 
$$d_{C(S)}(c_1,c_2)=d_{D(S)}(i(c_1),i(c_2)) \leq d_{\Delta(S)}(i(c_1),i(c_2)) \leq d_{C(S)}(c_1,c_2)~.$$ 
The image of $i$ is $1$-dense in $\Delta(S)$: every domain $X$ in $\Delta(S)$ admits an essential boundary component $x^b$, which actually determines an element of $i(C(S))$ at distance $1$. Hence, $i$ is a quasi-isometry.
\end{proof}

We remark that the composition of simplicial inclusions $C(S) \to \Delta(S) \to D(S)$ is the natural inclusion $C(S) \to D(S)$. By the above theorem, we have:
\begin{corollary}\label{COROLLARIO}
\begin{enumerate}
 \item Let $\Delta(S)$ be a subcomplex of $D(S)$ that contains $C(S)$ as a simplicial subcomplex. Then, the natural simplicial inclusion $\Delta(S) \to D(S)$ is a quasi-isometry.
\item Let $\Lambda(S)$ be a subcomplex of $D(S)$ and $\Lambda C(S)$ be the subcomplex of $D(S)$ spanned by the vertices of $\Lambda(S)$ and the vertices of $C(S)$. Then, the natural simplicial inclusion $\Lambda (S) \to \Lambda C(S)$ is a quasi-isometric embedding if and only if the natural simplicial inclusion $\Lambda(S) \to D(S)$ is a quasi-isometric embedding.   
\end{enumerate}

\end{corollary}

If $i: C(S) \to \Delta(S)$ is the above mentioned natural inclusion, we can exhibit an uncountable family of right inverse maps to $i$ which are quasi-isometries between $C(S)$ and $\Delta(S)$. \\
For every domain $X$, let us choose one of its essential boundary components, say $x^b$. Given any such choice, we now define a \emph{coarse projection} $\pi: \Delta(S) \to C(S)$ as the map such that $\pi(X)$ is the vertex in $C(S)$ given by $x^b$. Of course, by our definition, we have $\pi \circ i = id_{C(S)}$, and there exist infinitely many such coarse projections. We also notice that for every coarse projection $\pi$ and for every $X \in \Delta(S)$, we have $d_{\Delta(S)}(i \circ \pi(X),X ) \leq 1 $.

\begin{theorem}
The following statements hold:
\begin{enumerate}
\item Let $\pi_1, \pi_2: \Delta(S) \to C(S) $ be coarse projections. For every $X,Y \in \Delta(S)$ we have:
$$ d_{C(S)}(\pi_2(X), \pi_2(Y)) - 2 \leq d_{C(S)}(\pi_1(X),\pi_1(Y)) \leq d_{C(S)}(\pi_2(X),\pi_2(Y)) + 2 .$$
\item Let $\pi:\Delta(S) \to C(S)$ be a coarse projection. Then $\pi$ is a $(1,2)$-quasi isometric embedding and a quasi-isometry.
\end{enumerate}
\end{theorem}
 
\begin{proof}
Let us prove (1).\\
We notice that if $\pi_1(X) \neq \pi_2(X)$, there is an edge in $C(S)$ connecting them, for they are different boundary components of the same domain $X$. We get a path in $C(S)$ of vertices $\pi_1(X)\pi_2(X)\pi_2(Y)\pi_1(Y)$, and we can conclude that
\begin{align*}
d_{C(S)}(\pi_1(X), \pi_1(Y)) & \leq d_{C(S)}(\pi_1(X), \pi_2(X)) + d_{C(S)}(\pi_2(X), \pi_2(Y)) \\ 
			     & + d_{C(S)}(\pi_2(Y) , \pi_1(Y)) \\ 
			     &= d_{C(S)}(\pi_2(X),\pi_2(Y)) + 2 
\end{align*}
and 
\begin{align*}
d_{C(S)}(\pi_2(X), \pi_2(Y)) & \leq d_{C(S)}(\pi_2(X), \pi_1(X)) + d_{C(S)}(\pi_1(X), \pi_1(Y))  \\ 
			     & + d_{C(S)}(\pi_1(Y) , \pi_2(Y)) \\ 
			     & = d_{C(S)}(\pi_1(X),\pi_1(Y)) + 2
~. 
\end{align*}\\
Now we prove (2). \\
Let us consider the path given by the edge path $i(\pi(X))XYi(\pi(Y))$ in $\Delta(S)$ and remark that $d_{\Delta(S)}(i(\pi(X)), X), d_{\Delta(S)}(i(\pi(Y)),Y)$ are at most $1$.
By Theorem \ref{incl}, the simplicial inclusion $i: C(S) \to \Delta(S)$ is an isometric embedding, then 
$$d_{C(S)}(\pi (X),\pi (Y)) = d_{\Delta(S)}(i(\pi(X)), i(\pi(Y))) \leq d_{\Delta(S)}(X,Y) + 2$$ 
and $$d_{\Delta(S)} (X,Y) \leq d_{\Delta(S)} (i(\pi(X)), i(\pi(Y))) + 2 = d_{C(S)}(\pi (X),\pi(Y)) + 2~.$$
\end{proof}

\section{The arc complex $A(S)$ as a coarse subcomplex of $D(S)$}
Recall that an \emph{essential arc} on $S=S_{g,b}$ is a properly embedded arc whose endpoints are on the boundary components of $S$ such that it is not isotopic to a piece of boundary of $S$. The \emph{arc complex}\index{complex! of arcs} $A(S)$ is the simplicial complex whose $k$-simplices for $k \geq 0$ are collections of $k+1$ distinct isotopy classes of essential arcs on $S$ which can be realized disjointly. Recall that $A(S)$ is not a natural subcomplex of $D(S)$, unlike $C(S)$. The aim of this section is to prove that if $b \geq 3$ and $S \neq S_{0,4}$ (or, equivalently, $b \geq 3$ and $c(S) > 0$), then the arc complex $A(S)$ is quasi-isometric to the subcomplex of $D(S)$ whose vertices are pair of pants in $S$ having at least one boundary component on $\partial S$. 

Under the hypothesis on $S$, $A(S)$ is a locally infinite complex with infinitely many vertices, each maximal simplex corresponds to a decomposition of $S$ as union of hexagons, and it holds $ \mathrm{dim } A(S) = 6g + 3b - 7$ (for further details see \cite{MCP}).

In \cite{MC} Irmak and McCarthy prove that, except for a few cases, the group of simplicial automorphisms of $A(S)$ is the extended mapping class group of $S$. 

We regard $A(S)$ as a metric space with the natural metric such that every simplex is Euclidean, with edges of length 1. Like the above mentioned complexes, $A(S)$ is quasi-isometric to its 1-skeleton. 

\subsection{The boundary graph complex $A_B(S)$} 
Given an essential arc $\alpha$ on $S$, its \emph{boundary graph}\index{boundary graph} $G_\alpha$ is the graph obtained as the union of $\alpha$ and the boundary components of $S$ that contain its endpoints (see \cite{MCP}). 

\begin{definition}
The \emph{complex of boundary graphs}\index{complex! of boundary graphs} $A_B(S)$ is the simplicial complex whose $k$-simplices, for each $k \geq 0$, are collections of $k+1$ distinct isotopy classes of disjoint boundary graphs on $S$. 
\end{definition}

We shall always assume $S=S_{g,b}$ with $b\geq 3$ and $S\neq S_{0,4}$, for in these cases either $A_B(S)$ or $C(S)$ is not arcwise connected. \\
By identifying $G_\alpha$ with $\alpha$, we find that $A_B(S)$ and $A(S)$ have the same set of vertices, but in general $A_B(S)$ has fewer simplices than $A(S)$, for disjoint arcs with endpoints on the same boundary components are joined by an edge in $A(S)$ but not in $A_B(S)$. There is a natural simplicial inclusion between the complex of boundary graphs and the arc complex $A_B(S)\rightarrow A(S)~.$

We will always regard $A_B(S)$ and $A(S)$ as metric spaces with their natural simplicial metrics $d_{A_B(S)}$ and $d_{A(S)}$. 
We shall use the same notation for arcs on the surfaces and their isotopy classes on $A(S)$ or $A_B(S)$.

\begin{lemma}\label{conj}
Let $a,~b$ be two vertices of $A(S)$ that are joined by an edge in that complex. Consider now $a$ and $b$ as vertices of $A_B(S)$, Then $d_{A_B(S)}(a,b) \leq 4.$
\end{lemma}

\begin{proof}

By our assumption on $S$, either the genus $g$ of $S$ is $0$ and $S$ has more than $5$ boundary components, or the genus of $S$ is at least $1$ and $S$ has at least $3$ boundary components. Now, let $a,b$ be different vertices in $A(S)$, and assume they are not connected by an edge in $A_B(S)$. \\
In the case $g=0$, since $b \geq 5$, for every pair of vertices $a,b \in A(S)$ there exists a connected component of $S \setminus a \cup b$ which contains at least two different boundary components of $S$, and we can find a boundary graph that is disjoint from $a$ and $b$. Hence, the distance in $A_B(S)$ between $a$ and $b$ is 2. \\
For $g \geq 1$, we will give a detailed proof only in the case when $S$ has exactly $3$ boundary components, for this is the most restrictive case. We have different situations depending on the union of $a$ and $b$ as arcs intersecting minimally in the surface (here and in the rest of this proof, we shall consider the boundary components as marked points. This also the case in the figures below): 
\begin{enumerate}
\item $a \cup b$ is a simple closed curve. \\ It can bound a disc or not. In both cases, we have $d_{A_B(S)}(a,b) \leq 4$ (see Figures \ref{FIGURA1a} and \ref{FIGURA1b}, subcases (a) and (b)).
\begin{figure}[h!]
\centering%
{
\includegraphics[width=7cm]{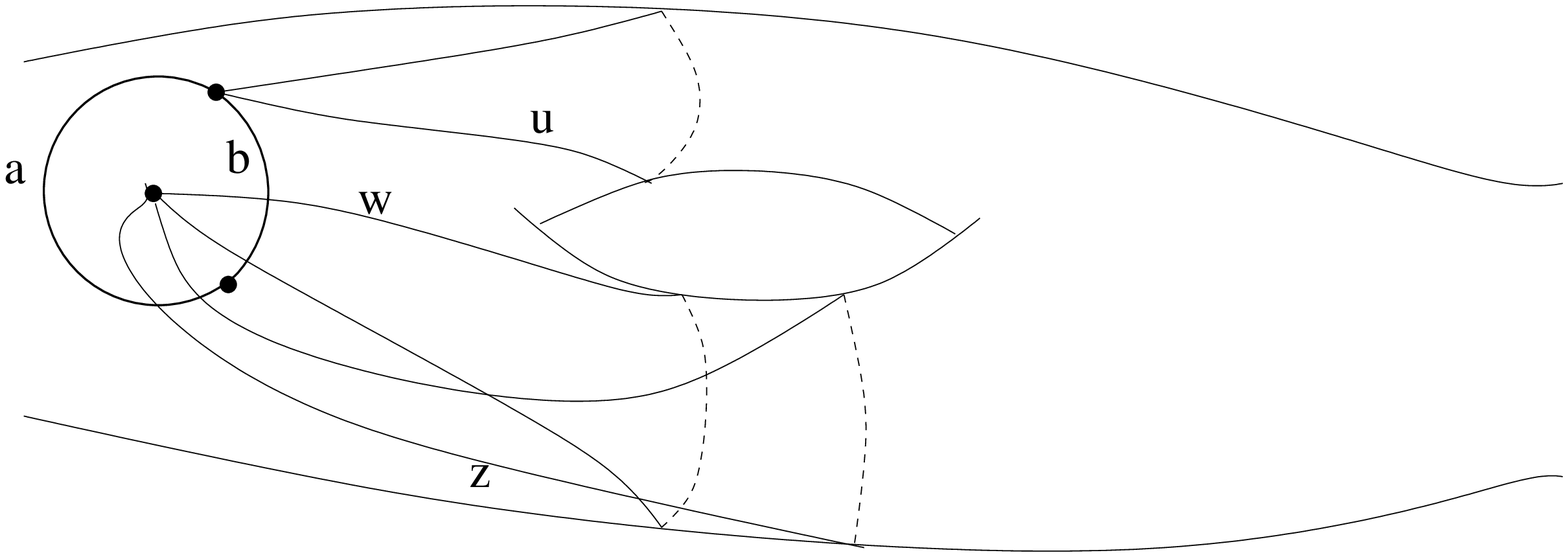}
\includegraphics[width=6cm]{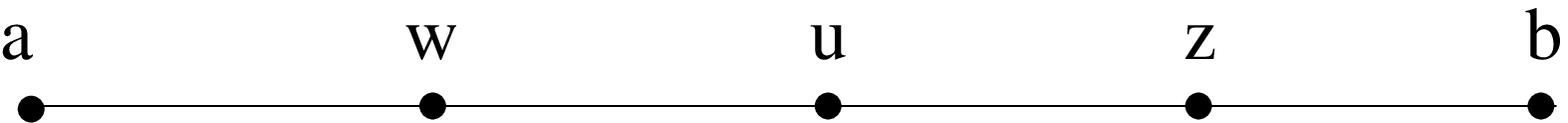}
}
\caption{The case where $a \cup b$ is a simple closed curve \label{FIGURA1a} (a)}
\end{figure}

\begin{figure}[h!]
\centering
{
\includegraphics[width=4cm]{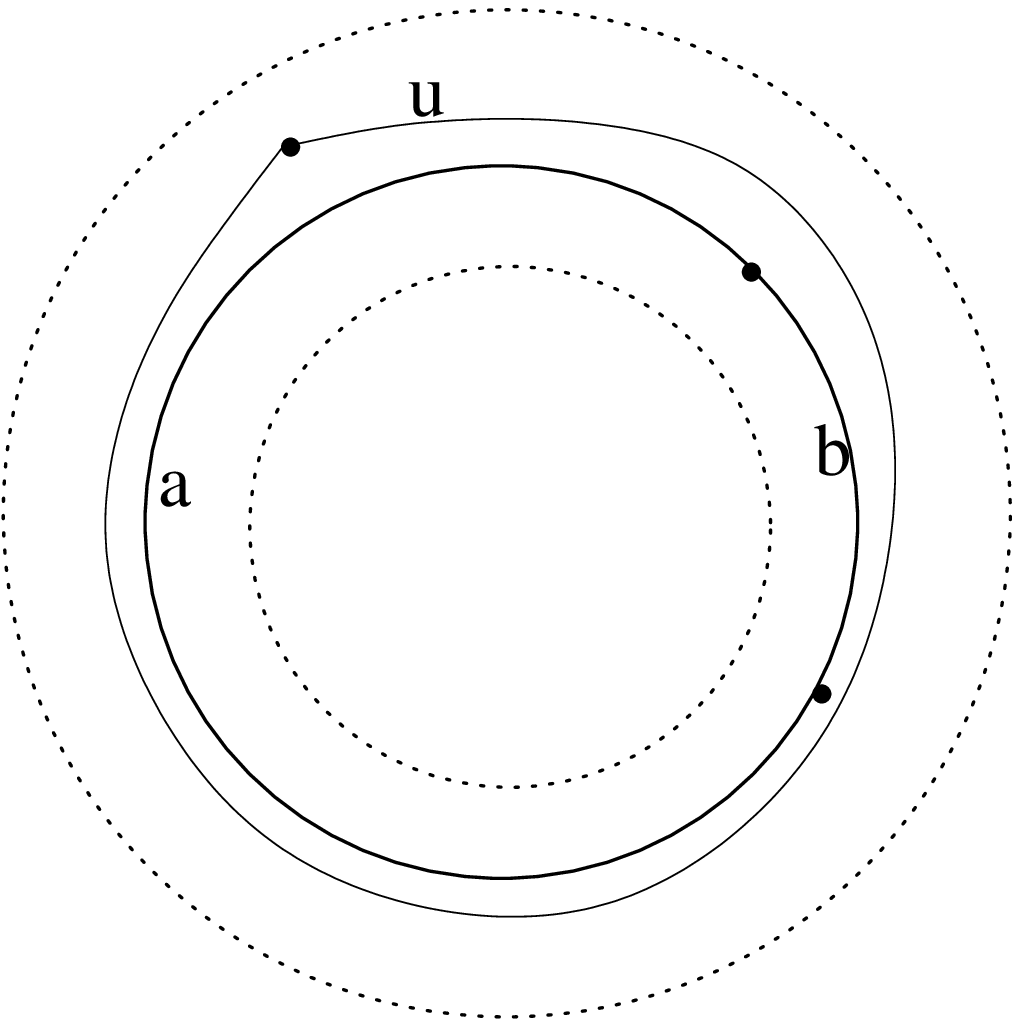}
\includegraphics[width=3.5cm]{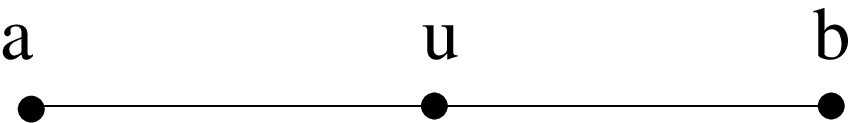}
}
\caption{The case where $a \cup b$ is a simple closed curve \label{FIGURA1b} (b)}
\end{figure}

\item $a \cup b$ is a simple arc with two different endpoints. \\
In this case, we have $d_{A_B(S)}(a,b) \leq 3$ (see Figure \ref{FIGURA2}).

\begin{figure}[htp]
\centering%
{
\includegraphics[width=6.5cm]{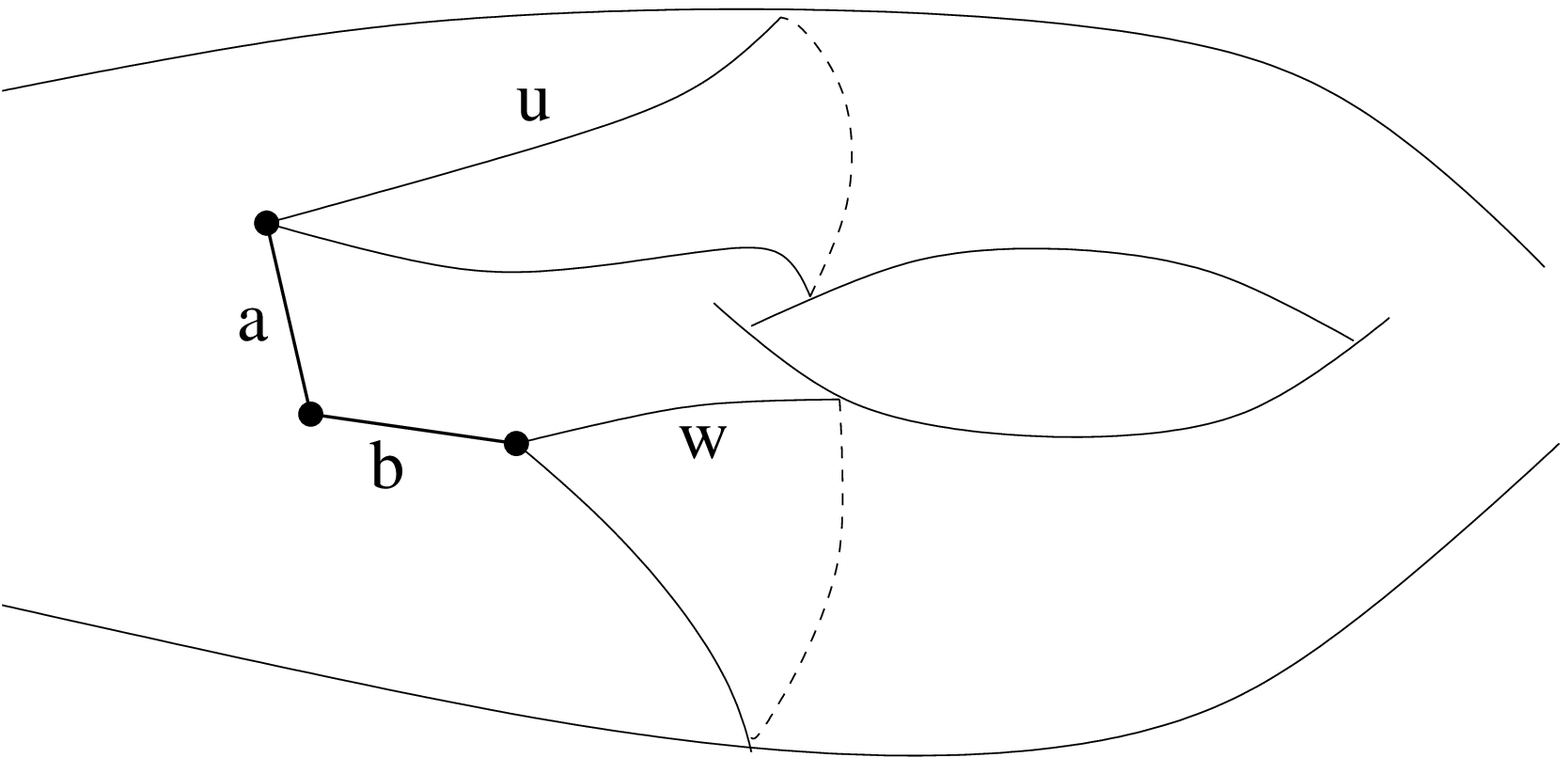}
\includegraphics[width=5cm]{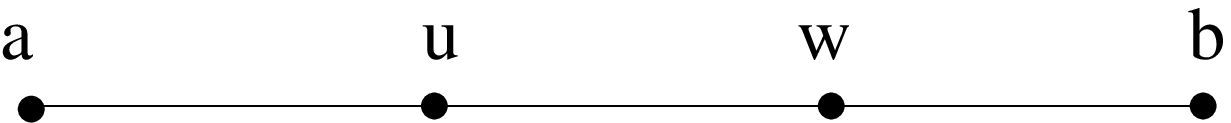}
}
\caption{The case where $a \cup b$ is an open arc \label{FIGURA2}}
\end{figure}

\item $a$ bounds a disc, and $b$ is not a closed curve. \\
In this case, we have $d_{A_B(S)}(a,b) \leq 4$ (see Figure \ref{FIGURA3a}, \ref{FIGURA3b}, \ref{FIGURA3c} subcases (a), (b), (c)).

\begin{figure}[h!]
\centering
{
\includegraphics[width=8.5cm]{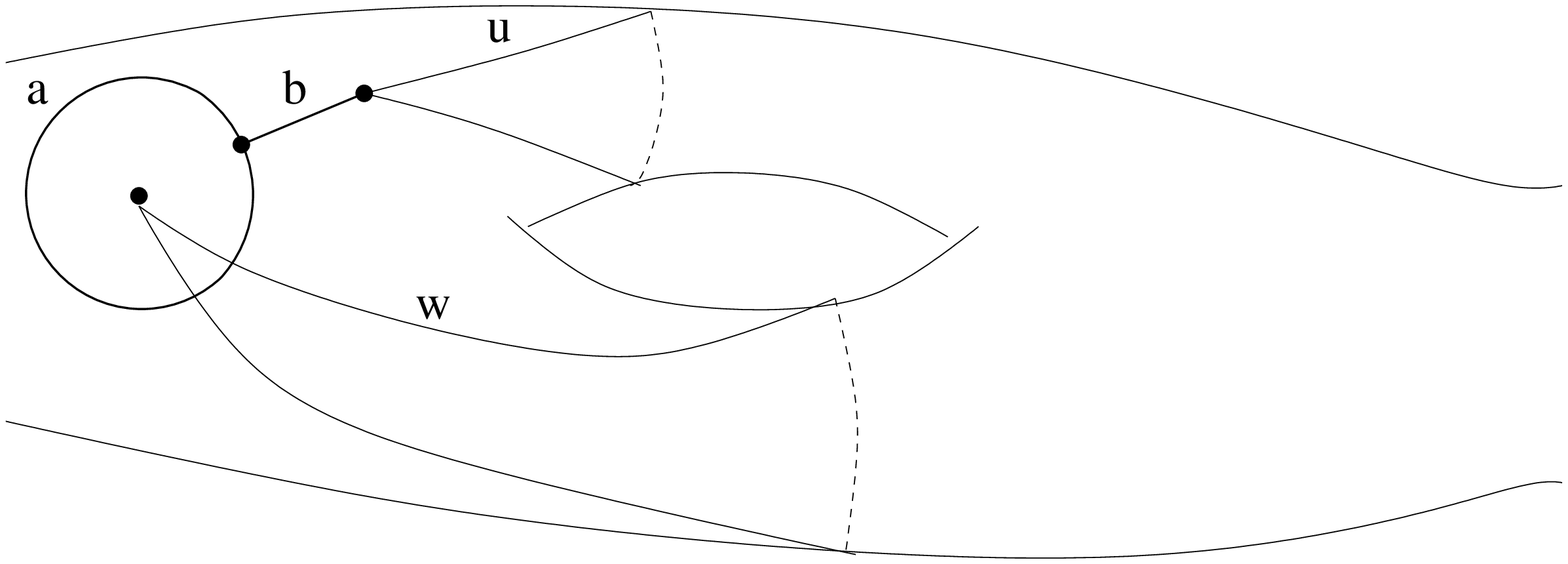}
\includegraphics[width=4.5cm]{3.eps}
}
\caption{The case where $a$ bounds a disc and $b$ is an open arc\label{FIGURA3a} (a)}
\end{figure}
\begin{figure}[h!]
\centering
{\includegraphics[width=8.5cm]{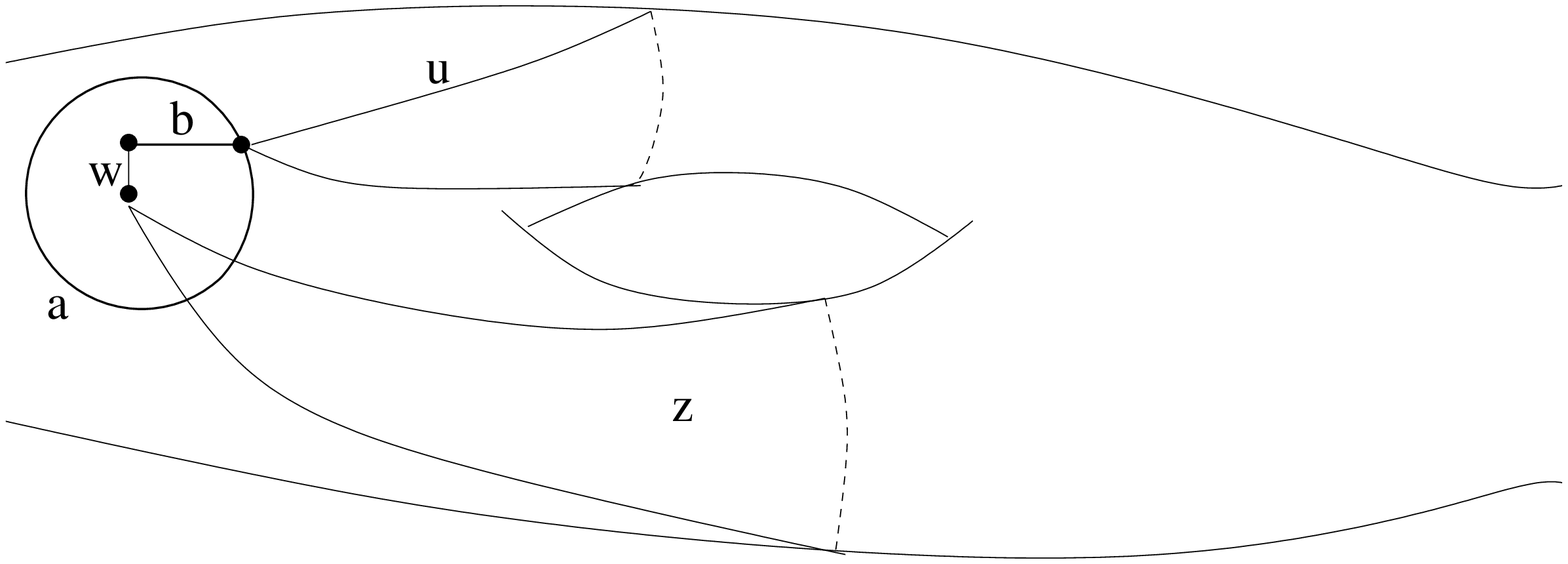}
\includegraphics[width=6cm]{4.eps}
}
\caption{The case where $a$ bounds a disc and $b$ is an open arc\label{FIGURA3b} (b)}
\end{figure}

\begin{figure}[h!]
\centering 
{\includegraphics[width=2cm]{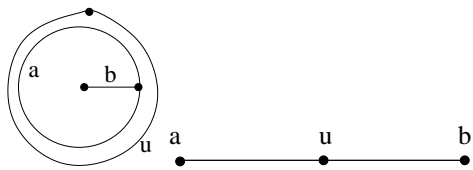}
\includegraphics[width=4cm]{2.eps}
}
\caption{The case where $a$ bounds a disc and $b$ is an open arc \label{FIGURA3c} (c)}
\end{figure}

\item Both $a$ and $b$ are closed simple curves, and $a$ bounds a disc. 
In this case, we have $d_{A_B(S)}(a,b) \leq 4$ (see Figure \ref{FIGURA4a}, \ref{FIGURA4b}, \ref{FIGURA4c}).

\begin{figure}[h!]
\centering%
{
\includegraphics[width=9cm]{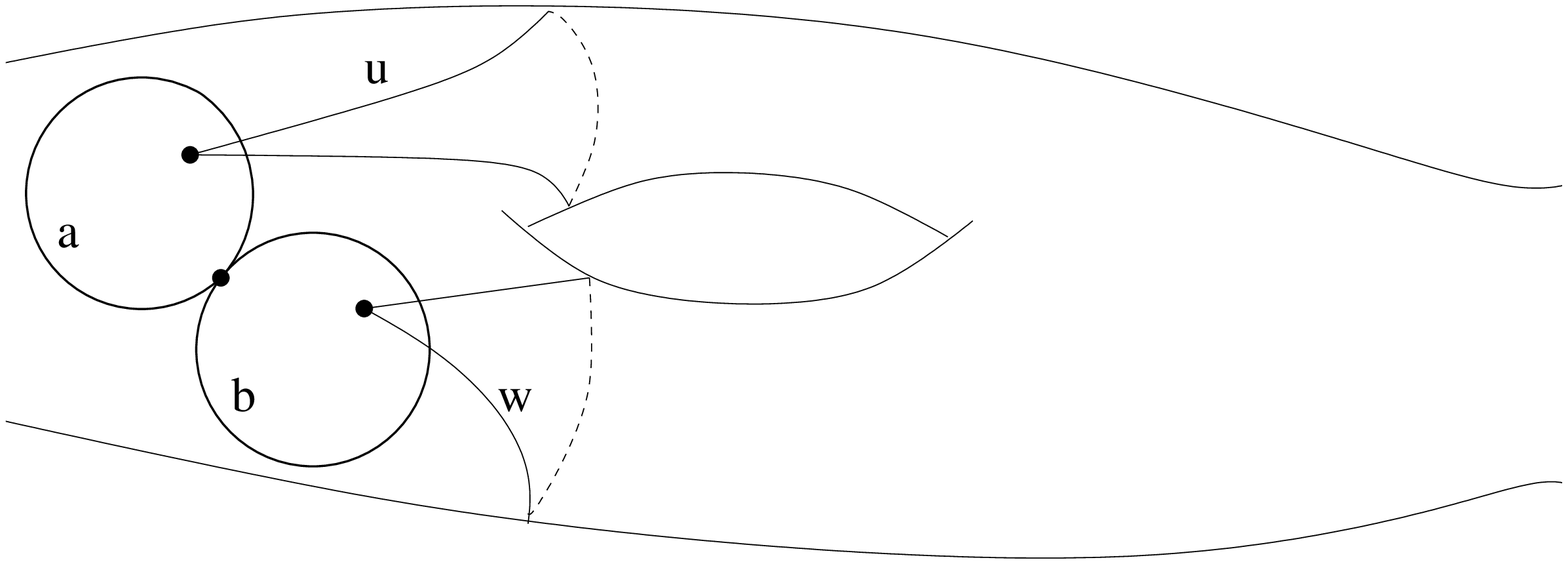}
\includegraphics[width=5cm]{3.eps}
}
\caption{The case where both $a$ and $b$ are closed curves, $a$ bounds a disc \label{FIGURA4a} (a)}
\end{figure}
\begin{figure}[h!]
\centering
{\includegraphics[width=9cm]{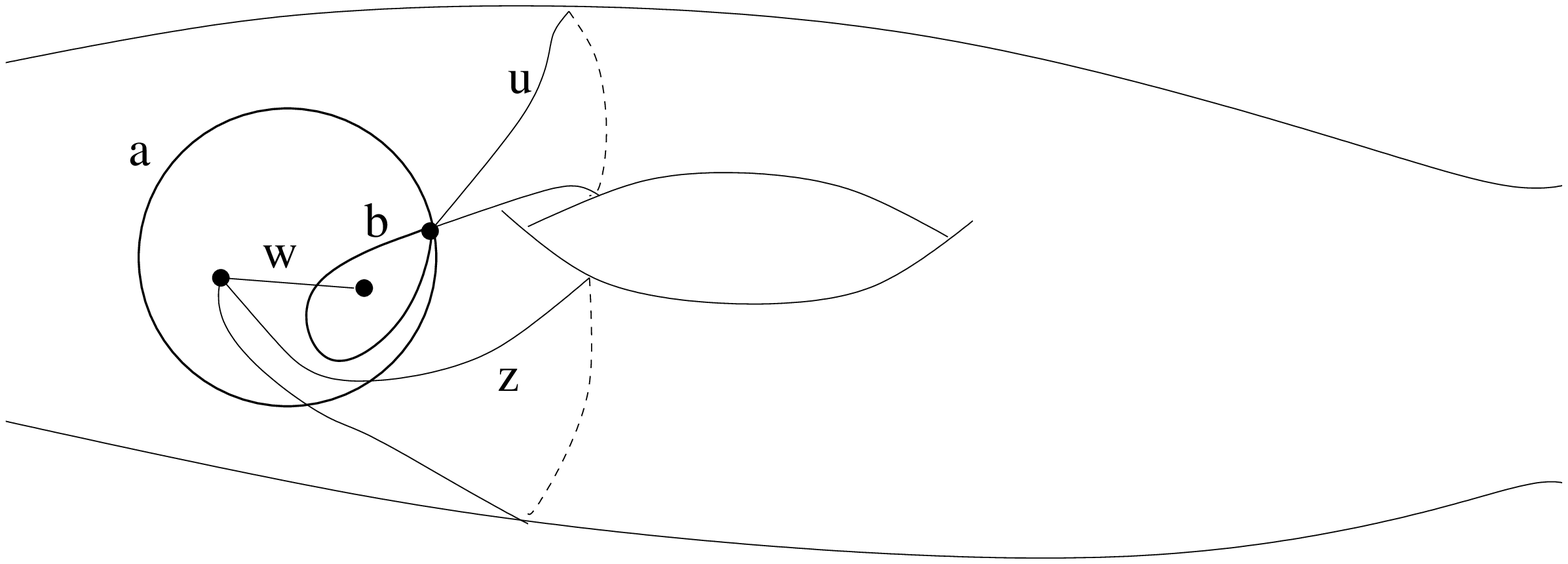}
\includegraphics[width=6cm]{4.eps}
}
\caption{The case where both $a$ and $b$ are closed curves, $a$ bounds a disc \label{FIGURA4b} (b)}
\end{figure}
\begin{figure}[h!]
\centering
{\includegraphics[width=4cm]{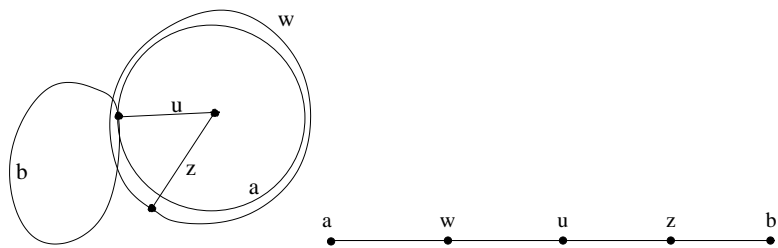}
\includegraphics[width=6cm]{4.eps}
}
\caption{The case where both $a$ and $b$ are closed curves, $a$ bounds a disc \label{FIGURA4c} (c)}
\end{figure}

\item  $a$ and $b$ are closed curves, but none of them bounds a disc. \\ 
Recall that $a$ and $b$ pass through the same boundary component of $\partial S$, and $a\cup b$ can disconnect $S$ in at most 3 distinct connected subsurfaces. Either one of them contains a simple arc joining the two remaining boundary components of $S$, or there is a non-disc component of $S \setminus a \cup b$ which contains an essential arc with both endpoints on the same boundary component of $\partial S$, disjoint from $a$ and $b$ (see Figure \ref{FIGURA5}). In this case, we get $d_{A_B(S)}(a,b) = 2$.
\item $a$ is a closed curve, which does not bound a disc and $b$ is not a closed curve. \\
We can use the same argument as in the previous case. 
\begin{figure}[h!]
\centering
\includegraphics[width=10cm]{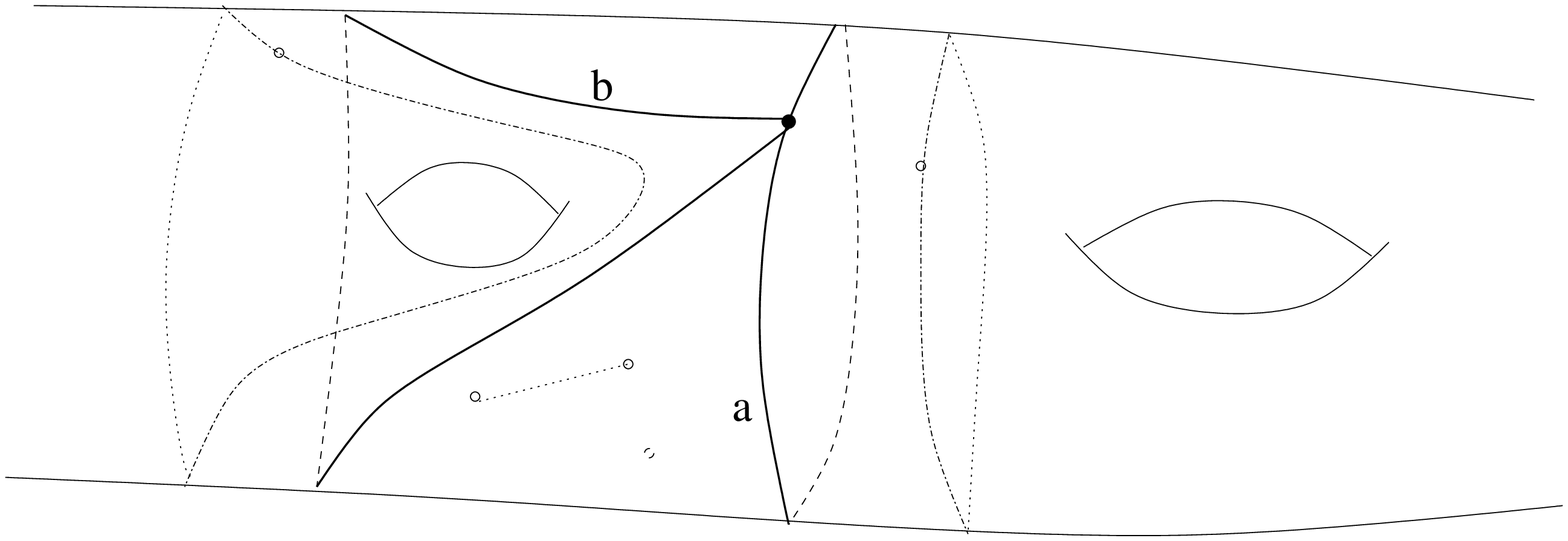} 
\includegraphics[width=3.5cm]{2.eps}
\caption{The case when both $a$ and $b$ are closed curves, but none of them bounds a disc \label{FIGURA5}}
\end{figure}
\end{enumerate}
\end{proof} 
At this point, we conclude the following result:
\begin{proposition}\label{arc}
The simplicial inclusion $j: A_{B}(S) \to A(S)$ induces a bilipschitz equivalence between $(A_{B}(S), d_{A_{B}(S)})$ and $(A(S),d_{A(S)})$. 
\end{proposition}

\begin{proof}
First, let us notice that for every pair of vertices $a_1, a_2$ in $A(S)$, we have $d_{A(S)}(a_1,a_2) \leq d_{A_B(S)}(a_1,a_2) ~.$ 

Let us consider a geodesic path $\sigma$ in $A(S)$ with endpoints $a_1,a_2$; say $\sigma$ is given by $a=x_0 x_1 \cdots x_n x_{n+1}=a_2$, where $x_0, \cdots x_{n+1}$ are vertices in $A(S)$. Consider the curve $\sigma^\sharp$ in $A_B(S)$ obtained by concatenating the geodesic segments $[x_i,x_{i+1}]$ in $A_B(S)$, say $\sigma^\sharp:~[x_0,x_1]*\cdots*[x_n,x_{n+1}]$. \\
By Lemma \ref{conj}, we have $ L_{A_B(S)}(\sigma^{\sharp}) \leq 4 L_{A(S)}(\sigma) ~.$
Of course, we also have 
$$ d_{A_B(S)}(a_1,a_2) \leq L_{A_B(S)}(\sigma^{\sharp}) \leq 4L_{A(S)}(\sigma) = 4 d_{A(S)}(a_1,a_2)~.$$
Hence, we get the following bilipschitz equivalence between the two distances: 
$$ d_{A(S)}(a_1,a_2) \leq d_{A_B(S)}(a_1,a_2) \leq 4 d_{A(S)}(a_1,a_2) ~.$$

\end{proof}

\subsection{$A_B(S)$ is quasi-isometric to $P_\partial(S)$}
A \emph{peripheral pair of pants}\index{pant!peripheral pair of} on $S$ is a pair of pants with at least one boundary component lying on $\partial S$. We define the complex of peripheral pair of pants $P_\partial(S)$ as the subcomplex of $D(S)$ induced by the vertices that are peripheral pairs of pants. A peripheral pair of pants is \emph{monoperipheral} if it has exactly one of its boundary components belonging to $\partial S$, otherwise it is \emph{biperipheral} (see \cite{MCP}). Any regular neighborhood of a boundary graph is a peripheral pair of pants.\\

Let us choose for every peripheral pair of pants $P$ an essential arc whose boundary graph has a regular neigbourhood isotopic to $P$. Any such choice determines a simplicial inclusion maps $i: P_\partial(S) \to A_B(S)$. Of course, there are infinitely many such maps. If $P$ is a monoperipheral pair of pants in $S$, there exists only one essential arc in $P$ whose boundary graph has a regular neighborhood isotopic to $P$. If $P$ is biperipheral, we can find $3$ such essential arcs (see Figure \ref{INJECT}). The path determined by the concatenation of these vertices has length $2$ in $A(S)$ (see Figure \ref{INJECT}) and length at most $8$ in $A_B(S)$ (by Lemma \ref{conj}).

\begin{figure}[hbp]
\centering
\includegraphics[width=4cm]{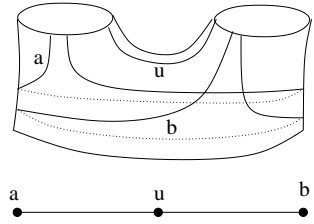}
\caption{Disjoint essential arcs with isotopic regular neighborhoods\label{INJECT}}
\end{figure}

\begin{proposition}\label{incl2}
The following statements hold:
\begin{enumerate}
\item Let $i_1, i_2: P_\partial(S) \to A_B(S)$ be two simpicial inclusions. For every $a,b \in P_\partial (S)$, the following inequalities hold 
$$d_{A_B(S)}(i_2(a),i_2(b)) - 16 \leq d_{A_B(S)}(i_1(a),i_1(b)) \leq d_{A_B(S)}(i_2(a),i_2(b)) + 16.$$
\item Any simplicial map $i: P_\partial (S) \to A_B(S)$ is an isometric embedding and a quasi-isometry. 
\end{enumerate}

\end{proposition}

\begin{proof}
Let us prove (1). \\
Let $a,b$ be two peripheral pairs of pants. By Lemma \ref{conj} we have the bounds ${d_{A_B(S)}(i_1(a),i_2(a)) \leq 8}$ and $d_{A_B(S)}(i_1(b),i_2(b)) \leq 8$.  
Looking at the quadrilateral with vertices $i_2(a)i_1(a)i_1(b)i_2(b)$, we get the following 
$$d_{A_B(S)}(i_2(a),i_2(b)) - 16 \leq d_{A_B(S)}(i_1(a),i_1(b)) \leq d_{A_B(S)}(i_2(a),i_2(b)) + 16 ~.$$
Let us prove (2). \\ 
From our definition, $i$ is surjective. We prove that it is an isometric embedding. 
If $P_1, P_2$ are disjoint peripheral pairs of pants, then their images are disjoint boundary graphs, and we have 
$d_{A_B(S)}(i(P_1),i(P_2)) \leq d_{P_\partial(S)}(P_1,P_2)~.$  
Now, let us consider the geodesic $\sigma$ in $A_B(S)$ given by the edge path $ \sigma: i(P_1)=b_0\cdots b_n= i(P_2)$. In a similar fashion, the condition that $b_i$ is disjoint from $b_{i+1}$ implies that the arcs have regular neighbourghoods that are disjoint from each other. Hence, the geodesic $\sigma$ projects to a curve $\sigma^\sharp$ on $A_B(S)$ given by isotopy classes of regular neighborhoods of the $b_i$'s, with endpoints are $P_1$ and $P_2$. We get 
$$d_{P_\partial(S)}(P_1,P_2) \leq L(\sigma^\sharp) \leq L(\sigma) = d_{A_B(S)}(i(P_1),i(P_2))~.$$ 
\end{proof}

Let us now consider the natural surjective map $\pi: A_B(S) \to P_\partial (S) $, which assigns to a boundary graph $a$ in $A_B(S)$ the isotopy class of the peripheral pair of pants given by a regular neighbourhood of $a$ in $S$. We notice that the map $\pi$ is not injective: any two vertices as in Figure \ref{INJECT} have the same image. It holds in general that if $d_{P_\partial(S)}(\pi b_1,\pi b_2) = 0$, then $d_{A_B(S)}(b_1,b_2) \leq 8$ (by Lemma \ref{conj}. For every simplicial inclusion map $i$ as in Proposition \ref{incl2}, we have $\pi \circ i = id_{P_\partial(S)}$ and $d_{A_B(S)}(i \circ \pi (x), x) \leq 4 d_{A(S)}(i \circ \pi (x), x)\leq 8$. 

We have the following:

\begin{proposition}\label{incl3}
The natural surjective map $\pi: A_B(S) \to P_\partial (S)$ is a $(1,8)$ quasi-isometric embedding and a quasi-isometry. 
\end{proposition}

\begin{proof}
If $b_1$ and $b_2$ are disjoint boundary graphs, then their regular neighborhoods are disjoint peripheral pairs of pants. Furthermore, if $P_1,P_2$ are disjoint pairs of pants, then one can realize disjointly every pair of boundary graphs $b_1,b_2$, whose regular neighbourhoods are $P_1,P_2$. Thus, we have the first inequality $d_{P_\partial(S)}(\pi b_1, \pi b_2) \leq d_{A_B(S)}(b_1,b_2)$. \\ 
If $\pi b_1 \neq \pi b_2$, let $\sigma$ be a geodesic in $P_\partial(S)$ defined by the concatenation of vertices $\sigma: \pi b_1 = P_1 \cdots P_n = \pi b_2$, with $P_i \cap P_{i+1} = \varnothing $. Choosing for every $P_i$ a boundary graph $b_i$, we get a curve $\sigma^\sharp$ in $A_B(S)$ given by the edge path $\sigma^\sharp: ~ b_1 \cdots b_n$, with $d_{A_B(S)}(b_i, b_{i+1}) =1 $. Hence, we get $d_{A_B(S)}(b_1,b_n)\leq L(\sigma^\sharp) = d_{P_\partial(S)}(\pi b_1, \pi b_2). $ \\
Finally, we get 
$$ d_{P_\partial (S)}(\pi b_1, \pi b_2) - 8 \leq d_{P_\partial(S)}(\pi b_1, \pi b_2) \leq d_{P_\partial (S)}(\pi b_1, \pi b_2) +8~. $$ 
\end{proof}

Using all the results proved in this section, we have:
\begin{theorem}\label{arc-per}
If $S=S_{g,b}$ such that $b \geq 3$ and $S\neq S_{0,4}$, the following holds: 
\begin{enumerate}
\item The complex of arcs $A(S)$ is quasi-isometric to the subcomplex $P_\partial(S)$ of $D(S)$.
\item If $g=0$, then the natural simplicial inclusion $P_\partial(S) \to D(S)$ is an isometric embedding and a quasi-isometry.
\item If $g \geq 1$, the image of the natural simplicial inclusion $k: P_\partial(S) \to D(S)$ is $2$-dense in $D(S)$, but the map $k$ is not a quasi-isometric embedding.
\end{enumerate}
\end{theorem}

\begin{proof}
The proof of (1) follows from the consideration that both the compositions $j\circ i,~j \circ \pi: A(S) \to P_\partial(S)$ of the maps in Lemma \ref{conj}, Proposition \ref{incl2} and  Proposition \ref{incl3} are quasi-isometries.
 
Let us prove (2). \\ 
Let $X$ be a domain. As a vertex of $D(S)$, $X$ is at distance 1 from each of the vertices representing its essential boundary components, and each essential boundary component of $X$ is at distance 1 from a pair of pants in $P_\partial(S)$. This shows that the image of the inclusion $P_\partial(S) \to D(S)$ is $2$-dense.\\ 
Recall that by our hypothesis on $S$, if $S=S_{0,b}$, then $b \geq 5$. Since the genus of $S$ is 0, each domain $X$ of $S$ is a sphere with holes, and each simple closed curve on $S$ disconnects the surface into two connected components, and each of them has at least one boundary component lying on $\partial S$. Let $P_1, P_2$ be peripheral pairs of pants of $S$, and $\gamma$ be a geodesic in $D(S)$ joining them, say $\gamma~:~P_1X_1\cdots X_{n-1}P_2$. \\
Let $\pi:D(S)\to C(S)$ be a coarse projection and $i:C(S) \to D(S)$ the natural simplicial inclusion. Notice that if $X_1$ is not a curve then $i(\pi(X_1))$ is a curve and it is disjoint from both $P_1$ and $X_2$. Moreover, $i(\pi(X_1))$ is also distinct from $X_2$ (otherwise we could shorten $\gamma$). Up to substituting $X_1$ with $i(\pi(X_1))$ and $X_{n-1}$ with $i(\pi(X_{n-1}))$, we can assume that both $X_1$ and $X_{n-1}$ are represented by simple closed curves. Moreover, up to substituting the segment of $\gamma$ given by $X_1\cdots X_{n-1}$ with the geodesic in $C(S)$ which joins $X_1$ and $X_{n-1}$ (see Theorem \ref{incl}), we can assume that each $X_i$ is a curve, say $C_i$.

If $\gamma$ has length $2$, it is represented by a path $P_1CP_2$. By geodesity, $P_1$ and $P_2$ belong to the same connected component of $S\setminus C$. Hence, there is a peripheral pair of pants $P^\star$ in the other connected component, and we find a new geodesic of length $2$ connecting $P_1,P_2$ contained in $P_\partial(S)$, namely $P_1P^\star P_2$.

If $\gamma$ has length greater than $2$, let us focus on the initial segment of $\gamma$ given by $P_1C_1C_2$. With the same argument, we can find a peripheral pair of pants $P^{c_1}$ disjoint from both $P_1$ and $C_2$, which can substitute $C_1$ in $\gamma$. Notice that the curve $P_1P^{c_1}C_2\cdots C_{n-1}P_2$ has the same lenght as $\gamma$, hence is a geodesic. 
Continuing this way, we get a path $\gamma^\star:~P_1P^{c_1}\cdots P^{c_{n-1}}P_2$ with all vertices in $P_\partial(S)$ with the same lenght as $\gamma$. Hence, we have
$$d_{D(S)}(P_1,P_2)\leq d_{P_{\partial(S)}}(P_1,P_2) \leq L_{P_{\partial(S)}}(\gamma^\star) = L_{D(S)}(\gamma^\star)=d_{D(S)}(P_1,P_2)~.$$

Let us prove (3). \\
As in the previous case, it is easy to see that the image of $k$ is $2$-dense. \\
 Let $c$ be a simple closed loop on $S$ wrapping around all the boundary components of $S$ as in Figure \ref{ultima}. Since $S$ is not a sphere, $c$ is essential, and disconnects the surface in two domains: one of them is the sphere $B=B_{0,b+1} \subset S$ which contains all the boundary components of $S$, the other is the subsurface $C=C_{g,1}$ which has $c$ as unique boundary component (see Figure \ref{ultima}). 
\begin{figure}[b]
\centering
\includegraphics[width=4cm]{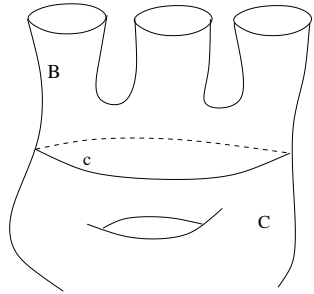}
\caption{The subsuface $C=C_{g,1}$ \label{ultima}}
\end{figure}
First, we claim that for every pair of peripheral pair of pants $P_1,P_2$ we have $d_{P_\partial(S)}(P_1,P_2) \geq d_{P_\partial(B)}(P_1,P_2)$. Let $\gamma$ be a geodesic in $P_\partial(S)$ joining $P_1,P_2$. If $Q$ is a vertex of $\gamma$, but not a peripheral pair of pants of $B$, then $Q$ crosses transversely a regular neighborhood of the curve $c$ and $Q \cap C$ is a strip. We can then replace this strip with one of the two strips of the neighborhood of $c$ in order to get a new peripheral pair of pants $Q' \in P_\partial (B)$, which can substitute $Q$ in $\gamma$. The curve obtained from $\gamma$ after all these substitutions may be shorter than $\gamma$, but all its vertices belong to $P_\partial(B)$. Hence, we have $d_{P_\partial(S)}(P_1,P_2) \geq d_{P_\partial(B)}(P_1,P_2)$.
Now, by the hypothesis $B$ has at least $4$ boundary components, hence by statement (2) we have $\mathrm{diam } ~P_\partial(B) = \mathrm{diam } ~D(B) = + \infty$. Hence, there exist $P_0, P_n$ peripheral pair of pants on $B$ such that $d_{P_\partial(B)}(P_0,P_n) \geq n$. Without loss of generality, we can also assume that all the boundary components of $P_0,P_n$ are boundary components of $S$. By our claim, we have $d_{A_\partial(S)}(P_0,P_n) \geq n$. Moreover, since both $P_0$ and $P_n$ are disjoint from $C$, we have $d_{D(S)}(P_0,P_n) = 2$, and this concludes the proof.
\end{proof}
As an immediate consequence of Theorem \ref{arc-per} and Theorem \ref{incl}, we have:
\begin{corollary}
If $S=S_{0,b}$ and $b \geq 5$, then $A(S)$ is quasi-isometric to $C(S)$. 
\end{corollary}

\section{Application: the arc and curve complex}
The \emph{arc and curve complex}\index{complex! arc and curve} $AC(S)$ is the simplicial complex whose $k$-simplices are collections of $k+1$ isotopy classes of essential arcs or curves on $S$. Its automorphism group is the extended mapping class group, and it is quasi-isometric to the curve complex (see \cite{KP}). 
As an application of Theorem \ref{arc-per} and Theorem \ref{incl}, we give a simple proof of the quasi-isometry between $C(S)$ and $AC(S)$ stated in \cite{KP} in the case $S=S_{g,b}$ with $b \geq 3$ and $S \neq S_{0,4}$. Moreover, we give necessary and sufficient conditions on $S$ for the natural inclusion $A(S) \to AC(S)$ to be a quasi-isometry. The latter result is also stated in \cite{MS}. 
\begin{theorem}
If $S=S_{g,b}$ with $b \geq 3$ and $S \neq S_{0,4}$, then the following holds: 
\begin{enumerate}
\item The arc and curve complex $AC(S)$ is quasi-isometric to $C(S)$. 
\item If $S$ is a sphere with boundary, the simplicial inclusion $A(S) \to AC(S)$ is a quasi-isometry. In the other cases, the simplicial inclusion is not a quasi-isometric embedding. 
\end{enumerate}
\end{theorem}

\begin{proof}
Let us prove (1). We consider the subcomplex of $D(S)$ spanned by the vertices of $C(S)$ and $P_\partial(S)$, say $P_\partial C(S)$, and the subcomplex of $AC(S)$ spanned by the vertices of $A_B(S)$ and $C(S)$, say $A_BC(S)$. We remark that:
\begin{description}
 \item{i.} $AC(S)$ is quasi-isometric to $A_BC(S)$, by a quasi-isometry given by the inclusion of $A_B(S)$ in $A(S)$ as in Lemma \ref{conj}; 
\item{ii.} $A_BC(S)$ is quasi-isometric to $P_\partial C(S)$, by a quasi-isometry induced by the natural isometric embedding $i: P_{\partial}(S) \to A_B(S)$ as in Proposition \ref{incl2}.
 \end{description} 
By Theorem \ref{incl} the complex $P_\partial C(S)$ is quasi-isometric to $C(S)$, and we conclude. 

Let us prove (2). By Corollary \ref{COROLLARIO} the simplicial inclusion $P_\partial C (S) \to D(S) $ is a quasi-isometry. \\ In the diagram below the horizontal rows are the natural simplicial inclusions, and the vertical rows are the above mentioned quasi-isometries. The diagram commutes, hence the natural inclusion $A(S) \to AC(S)$ is a quasi-isometric embedding if and only if the natural inclusion $P_\partial (S) \to P_\partial C (S)$ is a quasi-isometric embedding. Moreover, the latter holds if and only if the resulting inclusion map $P_\partial(S) \to D(S)$ is a quasi-isometric embedding. We then conclude using Theorem \ref{arc-per}.
$$\xymatrix{
A(S)  \ar[r] \ar[d]   & AC(S)  \ar[d]  \\
A_B(S) \ar[r] \ar[u]^{q.i.} \ar[d] & A_BC(S) \ar[u]^{q.i.} \ar[d] \\   
P_\partial(S) \ar[u]^{q.i} \ar[r] & P_\partial C(S) \ar[r]^{q.i.} \ar[u]^{q.i.} &D(S) \ar[l] \\
}$$
\end{proof}

Valentina Disarlo \\
Universit\'e de Strasbourg - IRMA \\
7, rue Ren\'e-Descartes \\ 
67084 Strasbourg Cedex (France) \\ \\
Scuola Normale Superiore di Pisa \\
Piazza dei Cavalieri, 7
56100 Pisa (Italy) \\
valentina.disarlo@sns.it

 \end{document}